\documentclass{amsart}

\usepackage{amsthm,amssymb}
\usepackage[shortlabels]{enumitem}
\usepackage{mathtools}
\usepackage[pdftitle={A geometric and game-theoretic study of the conjunction of possibility measures},%
pdfauthor={Enrique Miranda, Matthias C. M. Troffaes, Sebastien Destercke},%
pdfkeywords={possibility measure, conjunction, imprecise probability, game theory, natural extension, coherence}]{hyperref}
\usepackage[dvipsnames]{xcolor}
\usepackage{tikz}
\usetikzlibrary{matrix}

\newcommand{\pr}{P}

\newcommand{\apr}{Q}
\newcommand{\unex}{\overline{E}}
\newcommand{\upr}{\overline{\pr}}
\newcommand{\pspace}{\mathcal{X}}
\newcommand{\credalset}{\mathcal{M}}
\newcommand{\setofprobs}{\mathbb{P}}

\DeclareMathOperator*{\argmin}{argmin}
\DeclareMathOperator*{\argmax}{argmax}

\newtheorem{theorem}{Theorem}
\newtheorem{lemma}[theorem]{Lemma}
\newtheorem{proposition}[theorem]{Proposition}
\newtheorem{corollary}[theorem]{Corollary}
\newtheorem{example}[theorem]{Example}

\title[A study of the conjunction of possibility measures]{A geometric and game-theoretic study of the conjunction of possibility measures}
\author{Enrique Miranda}
\address{University of Oviedo,  Department of Statistics and Operations Research, Spain}
\email{mirandaenrique@uniovi.es}
\author{Matthias C. M. Troffaes}
\address{Durham University, Department of Mathematical Sciences, UK}
\email{matthias.troffaes@gmail.com}
\author{S\'{e}bastien Destercke}
\address{Universit{\'e} de Technologie de Compiegne, CNRS UMR 7253, France}
\email{sdestercke@gmail.com}


\begin{document}

\begin{abstract}
In this paper, we study the conjunction of possibility measures when
they are interpreted as coherent upper probabilities, that is, as
upper bounds for some set of probability measures. We identify
conditions under which the minimum of two possibility measures
remains a possibility measure. We provide graphical way to check
these conditions, by means of a zero-sum game formulation of the
problem. This also gives us a nice way to adjust the initial
possibility measures so their minimum is guaranteed to be a
possibility measure. Finally, we identify conditions under which the
minimum of two possibility measures is a coherent upper probability,
or in other words, conditions under which the minimum of two
possibility measures is an exact upper bound for the intersection of
the credal sets of those two possibility measures.
\end{abstract}

\maketitle

\section{Introduction}

\subsection{Possibility Measures: Why (Not)}

Imprecise probability models \cite{1991:walley} are useful in
situations where there is insufficient information to identify a
single probability distribution. Many different kinds of imprecise
probability models have been studied in the literature
\cite{1996:walley::uncertaintyinexpertsystems}.
It has
been argued that closed convex sets of probability measures, also
called \emph{credal sets}, provide a unifying framework for
many---if not most---of these models
\cite{1991:walley,2008:miranda::survey:lowprevs}.

A downside of using credal sets in their full generality is that
they can be computationally quite demanding, particularly in
situations that involve many random variables. Therefore, in
practice, it is often desirable to work with simpler models whose
practicality compensate their limited expressiveness.
\emph{Possibility measures}
\cite{1978:zadeh::possibility,1988:dubois:possibility,1997:decooman:possibility1,1999:decooman:aeyels::sup:pres:upp:prob}
are among the simplest of such models, and present a number of
distinct advantages:
\begin{itemize}
\item Possibility measures can be easily elicited from experts, either through linguistic assessments~\cite{2005:decooman::behavioural}
or through lower confidence bounds over nested sets~\cite{1995:sandri}.
\item Possibility distributions provide
compact and easily interpretable graphical representations.
\item In large models, when exact computations are costly,
possibility measures can be simulated straightforwardly through random sets~\cite{2009:alvarez}
(for example to propagate uncertainty through complex models~\cite{2007:baudrit}).
\item Lower and upper expectations induced by possibility measures can be computed exactly by Choquet integration~\cite[Section~7.8]{2014:troffaes:decooman::lower:previsions}.
\item When interpreted as sets of probability measures,
possibility measures have a limited number of extreme points~\cite{2003:miranda::extreme,2006:kroupa}.
Many inference algorithms, for instance many of those used in graphical models,
employ extreme point representations: using
possibility measures in such algorithms will reduce the computational effort
required.
\end{itemize}

An obvious disadvantage of using a family
of simpler models is that the family may not be rich enough to allow
certain standard operations. For instance, multivariate joint models
obtained from possibilistic marginals are usually not possibility
distributions~\cite{2003:miranda::epistemic}, hence
outer-approximating possibility measures have been
proposed~\cite{2013:troffaes:infsci:pbox:possib,2009:destercke} to
allow one to use the practical advantages of such models.

\subsection{Formulation of the problem}

In this paper, we focus exclusively on the \emph{conjunction} of two
models, that is, the intersection of two credal sets.
The conjunction is of interest, for instance, when
possibility measures have been elicited from different experts, and we want to
know which probability measures are compatible with the assessments
of all experts simultaneously.
As such,
the conjunction is a combination rule that aggregates
pieces of information consisting of
several inputs to the same problem.

Many combination rules for imprecise probability models
are discussed in the literature;
see for instance
\cite{1994:coolen,1998:moral:delsagrado:aggregation,1999:dubois,2008:hunter,2010:benavoli:aggregation,2004:decooman:troffaes::products:aggregation}.
In this paper, we define the conjunction of two possibility measures
as the upper envelope of the set of probability measures that are
compatible (i.e., dominated) by both.
The following questions arise:
\begin{itemize}
\item It may happen that there is no probability measure
that is compatible with both possibility measures,
in which case the conjunction does not exist.
In the language
of imprecise probability theory, this means that the conjunction
\emph{incurs sure loss}.
When does this happen?
\item Even when there is at least one probability
measure that is compatible with both possibility measures,
the upper envelope may not
be a possibility measure. In order words, it is not guaranteed that
the conjunction on possibility measures is \emph{closed} \cite{1999:dubois}.
When is the conjunction of two possibility measures
again a possibity measure?
If it is not, can we effectively approximate it by a possibility measure?
\item Finally, if the conjunction is a possibility measure,
can we express that possibility measure directly in terms of the two
possibility measures that we are starting from, without going through
their credal sets?
\end{itemize}

We will answer each of the questions above,
using the notions of avoiding sure loss, coherence and
natural extensions from the behavioural theory of imprecise
probabilities \cite{1991:walley}.
The main contributions of this paper are:
\begin{itemize}
\item From a theoretical viewpoint, we provide
sufficient and necessary conditions for the intersection to be again
a credal set that can be represented by a possibility measure
(Theorems~\ref{thm:minmax} and~\ref{thm:minmax:2}).
\item From a practical perspective, we derive from these conditions
correction strategies such that the intersection of the corrected
models is an outer-approximating possibility distribution
(Lemma~\ref{lem:possibonsingletons} and
Theorem~\ref{thm:notposswhennonzero}).
\end{itemize}
Interestingly, some of our results can be proven quite elegantly by
means of standard results from zero-sum game theory
(Theorem~\ref{thm:minmax:game}). This theory
also leads us to a graphical method
to check the conditions and to apply
the correction strategy (Section~\ref{sec:graphical-method-examples}).

\subsection{Related literature}

The literature on the conjunction of possibility measures
is somewhat scarce.
However, there are quite a few related results that have been proven
in the context of evidence theory, which from a formal point of view
includes possibility theory as a particular case.

The compatibility of two possibility measures, meaning that the intersection of their associated sets of probabilities is non-empty,
was characterised by Dubois and Prade in \cite{1992:dubois::upper}.
Related work for belief functions was done by Chateauneuf in
\cite{1994:chateauneuf::combination}.

With respect to the conjunction of two possibility measures
again being a possibility measure, a necessary
condition is the \emph{coherence} of the minimum of these two
possibility measures. This coherence was investigated by
Zaffalon and Miranda in \cite{2013:zaffalon::probtime}. We are not
aware of any necessary and sufficient conditions for the conjunction
determining a possibility measures, and the only existing results
are counterexamples showing that this need not be the case: see
\cite{1992:dubois::upper}, and also
\cite{1994:chateauneuf::combination} for the case of belief
functions.



A related problem that has received more attention
is the connection between conjunction operators of possibility
theory and the conjunction operators of evidence theory: for example
Dubois and Prade \cite{1988:dubois::comb} study how Dempster's rule
relate to possibilistic conjunctive operators, and Destercke and
Dubois~\cite{2011:destercke} relate belief function combinations to
the minimum rule of possibility theory.

\subsection{Structure of the paper}

The paper is organised as follows. Section~\ref{sec:notation}
presents the notation we use and the problem we propose to tackle,
namely the properties of the conjunction of two possibility
measures. We begin in Section~\ref{sec:asl-coherence} by providing
conditions for the intersection of the credal sets associated with
two possibility measures to be non-empty, which means that the
conjunction of the possibility measures avoids sure loss. Then we
investigate in which cases this conjunction is a coherent upper
probability, meaning that it is the upper envelope of a credal set
(namely, the intersection of the two credal sets determined by the
possibility measures).

As we shall see, the coherence of the conjunction of two possibility
measures does not guarantee it is a possibility measure itself. We
deal with this problem in Section~\ref{sec:upr:possib}, by studying
under which conditions the upper probability resulting from the minimum of
two possibility measures is again a possibility measure.
We also provide a graphical way to check these conditions
that we also use to propose some correction strategy, as well as some
illustrative and practical examples.

When this conjunction avoids sure loss but is not coherent, we can
always consider its natural extension, that corresponds to taking
the upper envelope of the intersection of the credal sets, and that
is the greatest coherent upper probability that is dominated by the
conjunction of the two possibility measures. In
Section~\ref{sec:unex:possib}, we consider the problem of
establishing when this natural extension is a possibility measure.
Section~\ref{sec:diagnosis-example} illustrates the usefulness of
our results on a medical diagnosis problem. We conclude the paper in
Section~\ref{sec:conclusions} with some additional comments and
remarks.

\section{Notation}
\label{sec:notation}

\subsection{Upper Probabilities, Conjunction, Possibility Measures}

Consider a possibility space $\pspace$.
In this paper, we assume that $\pspace$ is finite.
$\wp(\pspace)$ denotes the power set (set of all subsets) of $\pspace$.
A function $Q\colon\wp(\pspace)\to[0,1]$
is called a \emph{probability measure} \cite{1950:kolmogorov}
whenever $Q(A\cup B)=Q(A)+Q(B)$ for all $A$ and $B\subseteq\pspace$
such that $A\cap B=\emptyset$, and $Q(\pspace)=1$.
The set of all probability measures
is denoted by $\setofprobs$.

A function $\upr\colon\wp(\pspace)\to[0,1]$ is called an \emph{upper
probability} \cite{1981:walley:lowuppprobs,1991:walley}. We can
interpret $\upr(A)$ behaviourally as a subject's infimum acceptable
selling price for the gamble that pays $1$ if $A$ obtains, and $0$
otherwise \cite{1961:smith,1991:walley}. The \emph{credal set}
$\credalset$ induced by $\upr$ is defined as the set of probability
measures it dominates,
\begin{equation}\label{eq:credalset}
  \credalset\coloneqq
  \{Q \colon Q\in\setofprobs \;\wedge\; (\forall
  A\subseteq\pspace)(Q(A)\le\upr(A))\}.
\end{equation}
We say that $\upr$ \emph{avoids sure loss} when
its credal set is non-empty.
In this case,
the \emph{natural extension} $\unex$ of $\upr$
is defined as the upper envelope of its credal set,
that is
\begin{equation}\label{eq:unex}
  \unex(A)\coloneqq\max_{Q\in\credalset}Q(A) \ \text{ for every } A\subseteq\pspace.
\end{equation}
An upper probability is called \emph{coherent} if it coincides with
its natural extension, that is, if $\upr(A)=\unex(A)$ for all
$A\subseteq\pspace$. As a consequence, if $\upr$ avoids sure loss
then its natural extension is the greatest coherent upper
probability it dominates.
A coherent upper probability $\upr$ is always sub-additive:
$\upr(A\cup B)\leq\upr(A)+\upr(B)$
for any disjoint subsets $A$ and $B$ of $\pspace$.

The \emph{conjunction} \cite{1982:walley:aggregation}
of two upper probabilities $\upr_1$ and $\upr_2$
is defined as
\begin{equation}
  \upr(A)\coloneqq\min\{\upr_1(A),\upr_2(A)\}  \ \text{ for every } A\subseteq\pspace.
\end{equation}
It embodies the behavioural implications of both $\upr_1$ and
$\upr_2$. Unfortunately, even if both $\upr_1$ and $\upr_2$ are
coherent, the conjunction $\upr$ may not be coherent.  One can check
that the credal set of the conjunction of $\upr_1$ and $\upr_2$ is
the intersection of the credal sets of $\upr_1$ and $\upr_2$
\cite{1982:walley:aggregation}:
\begin{equation}
  \credalset=\credalset_1\cap\credalset_2.
\end{equation}
If $\credalset$ is non-empty, $\upr$ can be made coherent
through its natural extension.

In this paper, we will be interested in coherent upper probabilities
of a very specific form. A function $\pi\colon\pspace\to[0,1]$ is
called a (normalized) \emph{possibility distribution}
\cite{1978:zadeh::possibility,1988:dubois:possibility,1997:decooman:possibility1,1982:giles:semantics}
whenever
\begin{equation}
  \max_{x\in\pspace}\pi(x)=1.
\end{equation}
A possibility distribution $\pi$ induces a \emph{possibility
measure} $\Pi\colon\wp(\pspace)\to[0,1]$ by
\begin{equation}\label{eq:poss-from-dist}
  \Pi(A)\coloneqq\max_{x\in A}\pi(x) \ \text{ for every } A\subseteq\pspace.
\end{equation}
A possibility measure is a coherent upper probability
\cite[p.~37]{1996:walley::uncertaintyinexpertsystems}.

\subsection{Conjunction of Two Possibility Measures}

Consider two possibility distributions $\pi_1$ and $\pi_2$ that
induce possibility measures $\Pi_1$ and $\Pi_2$, with associated
credal sets $\credalset_1$ and $\credalset_2$.
As just mentioned, the conjunction of these
two possibility measures is the upper envelope of
$\credalset=\credalset_1\cap\credalset_2$, and is denoted by
$\unex$. Alternatively, $\unex$ is the most conservative (i.e.
pointwise largest) coherent upper prevision which is dominated by
the upper probability $\upr$ defined by
\begin{equation}
  \upr(A)\coloneqq\min\{\Pi_1(A),\Pi_2(A)\}
\end{equation}
for all events $A\subseteq\pspace$.
Throughout the entire paper,
we will use the symbols $\pi_1$, $\pi_2$, $\Pi_1$, $\Pi_2$,
$\credalset_1$, $\credalset_2$, $\credalset$, $\upr$, and $\unex$,
always as defined in this section.

Note that, in general $\upr$ may not avoid sure loss (in which case
the conjunction does not exist), or may be incoherent (in which case
$\upr$ does not coincide with $\unex$), and even when it is
coherent, it may not be a possibility measure itself. In this paper,
we investigate in detail each of these cases, by providing necessary
and sufficient conditions for $\upr$ to satisfy each of these
properties.

\section{Avoiding sure loss and coherence}\label{sec:asl-coherence}

We begin by investigating under which conditions the upper
probability determined by the conjunction of two possibility
measures avoids sure loss or is coherent. These are the minimal
behavioural conditions established by Walley in \cite{1991:walley}.

\subsection{When does $\upr$ avoid sure loss?}

It is not difficult to show that $\upr$ does not avoid sure loss in
general.

\begin{example}
Let $\pspace=\{1,2\}$ and
\begin{center}
\begin{tabular}{c|cc}
   & 1 & 2 \\\hline
  $\pi_1$ & 1 & 0.3 \\
  $\pi_2$ & 0.5 & 1
\end{tabular}
\end{center}
Then any probability measure
$Q\in\credalset_1\cap\credalset_2$ must satisfy
$Q(\{1\})\leq 0.5$ and $Q(\{2\})\leq 0.3$.
This is incompatible with
$1=Q(\{1,2\})=Q(\{1\})+Q(\{2\})$, and therefore
$\credalset_1\cap\credalset_2=\emptyset$.
\end{example}

The following theorem, proven by Dubois and Prade~\cite[Lemma
5]{1992:dubois::upper}, gives a necessary and sufficient condition
for the upper probability $\upr$ to avoid sure loss:

\begin{theorem}\cite{1992:dubois::upper}\label{th:charac-asl}
$\upr$ avoids sure loss if and only if for all $A\subseteq\pspace$
\begin{equation}
 1 \leq \Pi_1(A)+\Pi_2(A^c).
\end{equation}
\end{theorem}

This result was also established for belief functions by Chateauneuf
in \cite{1994:chateauneuf::combination}, who refers to the non-empty
intersection of the credal sets as the \emph{compatibility} of their
associated imprecise probability models; see also
\cite{1987:decampos}. Other characterisations of avoiding sure loss
for the conjunction of possibility measures can be found in
\cite[Propositions~6 and~7]{1992:dubois::upper}.

\subsection{When is $\upr$ coherent?}

Recall that $\upr$ is coherent if and only if it coincides with its
natural extension $\unex$, that is, if and only if it coincides with
the upper envelope of its credal set $\credalset$, as in
Eq.~\eqref{eq:unex}. The conjunction $\upr$ can be incoherent even
if it avoids sure loss, as the following example shows:

\begin{example}
Let $\pspace=\{1,2,3\}$ and
\begin{center}
\begin{tabular}{c|ccc}
   & 1 & 2 & 3 \\\hline
  $\pi_1$ & 1 & 0.3 & 0.5 \\
  $\pi_2$ & 0.5 & 1 & 0.7
\end{tabular}
\end{center}
Then every probability
measure $Q\in\credalset=\credalset_1\cap\credalset_2$ must satisfy
$Q(A)\le\upr(A)=\min\{\Pi_1(A),\Pi_2(A)\}$ for all $A\subseteq\pspace$.
In particular,
\begin{align}
 Q(\{1\})&\leq 0.5, & Q(\{2\})&\leq 0.3, & Q(\{3\})&\leq 0.5,\\
 Q(\{1,2\})&\leq 1, & Q(\{1,3\})&\leq 0.7, & Q(\{2,3\})&\leq 0.5.
\end{align}
Since $Q(\{1\})\leq 0.5$ and $Q(\{2\})\leq 0.3$ imply that
$Q(\{1,2\})\leq 0.8$, but on the other hand $\upr(\{1,2\})=1$, it
follows that $\upr$ is incoherent. Still, $\upr$ avoids sure loss
because $\credalset$ contains the probability measure $Q$ with
$Q(\{1\})=0.5$, $Q(\{2\})=0.3$, and $Q(\{3\})=0.2$.
\end{example}

Given a credal set $\credalset$, the upper envelope of the set of
expectation operators with respect to the elements of $\credalset$
is called a \emph{coherent upper prevision}. The conjunction of two
coherent upper previsions with respective credal sets $\credalset_1$
and $\credalset_2$ is coherent if and only if
$\credalset_1\cup\credalset_2$ is convex
\cite[Theorem~6]{2013:zaffalon::probtime}. From the proof of
\cite[Theorem~6,
(a)$\Rightarrow$(b)$\Rightarrow$(c)]{2013:zaffalon::probtime}, one
can easily see that convexity of $\credalset_1\cup\credalset_2$ is
still sufficient (but not necessary) for the conjunction of two
upper probabilities on events to be coherent. This leads immediately
to the following sufficient condition for the coherence of $\upr$:

\begin{proposition}\label{pr:charac-coher-prevs}
$\upr$ is coherent if $\credalset_1\cup\credalset_2$ is convex.
\end{proposition}

The convexity of $\credalset_1\cup\credalset_2$ can
be checked in polynomial time \cite{2001:bemporad::convexity}.
The following example shows that convexity of
$\credalset_1\cup\credalset_2$ is not necessary for $\upr$ to be
coherent. It simultaneously shows that $\upr$ does not need to be a
possibility measure, even if it is coherent.

\begin{example}\label{ex:coh-not-possib}
Let $\pspace=\{1,2,3\}$ and
\begin{center}
\begin{tabular}{c|ccc}
   & 1 & 2 & 3 \\\hline
  $\pi_1$ & 1 & 0.5 & 0.5 \\
  $\pi_2$ & 0.5 & 1 & 0
\end{tabular}
\end{center}
Then $\upr$ is the probability measure with probability mass
function $(0.5,0.5,0)$. This is not a
possibility measure, but it is a coherent upper probability (because
it is trivially the upper envelope of itself).

Also,
$\credalset_1\cup\credalset_2$ is not convex.
Using vector notation for probability mass functions,
we have that
\begin{equation}
 (0.5,0.25,0.25)\in\credalset_1
 \text{ and } (0.25,0.75,0)\in\credalset_2
\end{equation}
but their average $(0.375,0.5,0.125)$ does not belong to
$\credalset_1\cup\credalset_2$, because
\begin{equation}
Q(\{2,3\})=0.625>0.5=\Pi_1(\{2,3\})
\text{ and }
Q(\{3\})=0.125>0=\Pi_2(\{3\}).
\end{equation}
Indeed, that $Q(A) > \Pi_i(A)$ for some event $A$ implies that $Q \not \in \credalset_i$.
\end{example}

Regarding \cite[Theorem~6]{2013:zaffalon::probtime}, let $\upr_1$
and $\upr_2$ denote the upper envelopes of the sets of expectation
operators with respect to the credal sets $\credalset_1$ and
$\credalset_2$ in this example. Then the conjunction
$\min\{\upr_1,\upr_2\}$ is not equal to the expectation operator
associated with $(0.5,0.5,0)=\credalset_1\cap\credalset_2$. To see
this, consider the gamble $f$ given by $f(1)=1$, $f(2)=2$, and
$f(3)=3$. For this gamble, $Q(f)=1.5<2=\min\{\upr_1(f),\upr_2(f)\}$.

Next we show that the minimum $\upr$ of two possibility measures can
be a coherent upper probability that is not even $2$-alternating (and
thus not a possibility measure, either).
 Recall that $\upr$ is $2$-alternating if $\upr(A) + \upr(B) \leq \upr(A \cup B) + \upr(A \cap B)$
 for any $A,B \subseteq \pspace$.

\begin{example}\label{ex:coh-not-2alt}
Let $\pspace=\{1,2,3,4\}$ and
\begin{center}
\begin{tabular}{c|cccc}
   & 1 & 2 & 3 & 4 \\\hline
  $\pi_1$ & 0.3 & 0.4 & 0.6 & 1 \\
  $\pi_2$ & 0.3 & 0.6 & 0.4 & 1
\end{tabular}
\end{center}
It can be shown by linear programming that $\upr$ is coherent.
However, it is not $2$-alternating: for $A=\{1,2\}$ and $B=\{1,3\}$,
it holds that
\begin{align}
 \upr(A\cup B)+\upr(A\cap B)
 &=\upr(\{1,2,3\})+\upr(\{1\})=0.6+0.3=0.9\\
 &>\upr(A)+\upr(B)=\upr(\{1,2\})+\upr(\{1,3\})=0.8.
\end{align}
\end{example}

The following result is rather surprising:
we can show that the conjunction $\upr$ of two possibility measures
is $2$-alternating when $\credalset_1\cup\credalset_2$ is convex;
it strengthens Proposition~\ref{pr:charac-coher-prevs}.

\begin{proposition}\label{pr:convexity-2alt}
$\upr$ is $2$-alternating if $\credalset_1\cup\credalset_2$ is convex.
\end{proposition}
\begin{proof}
By \cite[Corollary~6.4]{1981:walley:lowuppprobs}, to show that
$\upr$ is $2$-alternating, it suffices to establish that for every
$A\subseteq B\subseteq \pspace$ there is a $Q\in\credalset$ such
that $Q(A)=\upr(A)$ and $Q(B)=\upr(B)$.

Consider $A\subseteq B\subseteq\pspace$.
Because $\Pi_1$ is a possibility measure and
therefore $2$-alternating, there is a $Q_1\in\credalset_1$ such
that $Q_1(A)=\Pi_1(A)$ and $Q_1(B)=\Pi_1(B)$.
Similarly, there is a $Q_2\in\credalset_2$ such that
$Q_2(A)=\Pi_2(A)$ and $Q_2(B)=\Pi_2(B)$. Now, since
$\credalset_1\cup\credalset_2$ is convex, it follows from
\cite[Theorem~6]{2013:zaffalon::probtime} that there is an
$\alpha\in[0,1]$ such that $Q\coloneqq\alpha Q_1+(1-\alpha)Q_2$
belongs to $\credalset_1\cap\credalset_2=\credalset$, and as a
consequence $Q$ is dominated by $\upr$:
\begin{align}
  Q(A)&\le\upr(A) & Q(B)&\le\upr(B).
\end{align}
But, by construction of $Q$, we also have that
that
\begin{align}
 \nonumber
 Q(A)&=\alpha Q_1(A)+(1-\alpha) Q_2(A) \\
 &\geq\min\{Q_1(A),Q_2(A)\}=\min\{\Pi_1(A),\Pi_2(A)\}=\upr(A)
 \label{eq:2alt:proof:qa}
 \\
 \nonumber
 Q(B)&=\alpha Q_1(B)+(1-\alpha) Q_2(B) \\
 &\geq\min\{Q_1(B),Q_2(B)\}=\min\{\Pi_1(B),\Pi_2(B)\}=\upr(B).
 \label{eq:2alt:proof:qb}
\end{align}
Concluding, $Q(A)=\upr(A)$ and $Q(B)=\upr(B)$,
so $\upr$ is $2$-alternating.
\end{proof}

To see that the convexity of $\credalset_1\cup\credalset_2$ does not
guarantee that $\upr$ is a possibility measure, consider the
following example:

\begin{example}\label{ex:convex-not-possib}
Let $\pspace=\{1,2\}$ and
\begin{center}
  \begin{tabular}{c|cc}
    & 1 & 2 \\
    \hline
    $\pi_1$ & 0.5 & 1 \\
    $\pi_2$ & 1 & 0.5
  \end{tabular}
\end{center}
Then $\upr$ is the probability measure
determined by the probability mass function $(0.5,0.5)$,
which is obviously not a possibility measure.
However,
$\credalset_1$ is the set of all probability measures
$Q$ for which $Q(\{x_1\})\le 0.5$,
and
$\credalset_2$ is the set of all probability measures
$Q$ for which $Q(\{x_1\})\ge 0.5$,
so
$\credalset_1\cup\credalset_2$
is the set of all probability measures on $\pspace$,
which is convex.
\end{example}

From the proof of Proposition~\ref{pr:convexity-2alt}, we see that
the convexity of $\credalset_1\cup\credalset_2$ is actually a really
strong requirement.
Specifically, it requires that, for all $A\subseteq B$,
\begin{gather}
  \Pi_1(A)<\Pi_2(A)\implies\Pi_1(B)\le \Pi_2(B)
  \\
  \Pi_1(A)>\Pi_2(A)\implies\Pi_1(B)\ge \Pi_2(B)
  \\
  \Pi_1(B)<\Pi_2(B)\implies\Pi_1(A)\le \Pi_2(A)
  \\
  \Pi_1(B)>\Pi_2(B)\implies\Pi_1(A)\ge \Pi_2(A)
\end{gather}
Indeed,
if $\credalset_1\cup\credalset_2$ is convex,
following the proof of Proposition~\ref{pr:convexity-2alt},
taking Eqs.~\eqref{eq:2alt:proof:qa} and~\eqref{eq:2alt:proof:qb}
and noting that $Q_i(A)=\Pi_i(A)$ and $Q_i(B)=\Pi_i(B)$,
we know that there is an $\alpha\in[0,1]$ such that
\begin{align}
  \alpha \Pi_1(A)+(1-\alpha) \Pi_2(A)
  &=\min\{\Pi_1(A),\Pi_2(A)\}
  \\
  \alpha \Pi_1(B)+(1-\alpha) \Pi_2(B)
  &=\min\{\Pi_1(B),\Pi_2(B)\}
\end{align}
So, if $\Pi_1(A)<\Pi_2(A)$,
then it must be that $\alpha=1$
by the first equality,
and therefore also $\Pi_1(B)\le\Pi_2(B)$
by the second equality.
The other cases follow similarly.

These implications give us a simple way to check
for typical violations of convexity of
$\credalset_1\cup\credalset_2$, through the following corollary.

\begin{corollary}\label{cor:convex-is-pretty-strong}
  If $\credalset_1\cup\credalset_2$ is convex,
  then for all subsets $A$, $B$, and $C$ of $\pspace$
  such that
  $\Pi_1(A)<\Pi_2(A)$, $\Pi_1(B)>\Pi_2(B)$,
  and $C\supseteq A\cup B$,
  we have that $\Pi_1(C)=\Pi_2(C)$.
\end{corollary}

In a way, Example~\ref{ex:convex-not-possib} is thus showing a very
peculiar situation (corresponding to $A=\{x_1\}$, $B=\{x_2\}$, and
$C=\{x_1,x_2\}$ in Corollary~\ref{cor:convex-is-pretty-strong}).

One of the advantages of possibility measures over other imprecise
probability models is their computational simplicity, that follows
from Eq.~\eqref{eq:poss-from-dist}: possibility measures are
uniquely determined by their restriction to singletons, called their
possibility distributions. Moreover, possibility distributions
connect possibility measures with fuzzy sets
\cite{1978:zadeh::possibility}. The minimum of two possibility
distributions was defined by Zadeh as one instance of \emph{fuzzy
set intersection}. However, the connection between imprecise
probabilities and fuzzy sets by means of possibility measures does
not hold under the conjunction operator we are considering in this
paper, in the sense that, as we have seen in
Example~\ref{ex:coh-not-2alt}, coherent conjunctions of possibility
measures need not be determined by their restrictions to singletons.
One might wonder if these restrictions suffice to characterise the
coherence of $\upr$. Clearly, a necessary condition for the
coherence of $\upr$ is that for every $x\in\pspace$ there is some
$Q\in\credalset_1\cap\credalset_2$ such that $Q(\{x\})=\upr(\{x\})$.
However, this condition is not sufficient, as the following example
shows.

\begin{example}
Let $\pspace=\{1,2,3\}$ and
\begin{center}
\begin{tabular}{c|ccc}
   & 1 & 2 & 3 \\\hline
  $\pi_1$ & 0.8 & 0.2 & 1 \\
  $\pi_2$ & 0.2 & 0.9 & 1
\end{tabular}
\end{center}
Then $(0,0,1)$ belongs to $\credalset_1\cap\credalset_2$, so $\upr$
avoids sure loss. However, it is not coherent because
$\upr(\{1,2\})=0.8>\upr(\{1\})+\upr(\{2\})=0.4$.

One can easily check that both
$(0.2,0.2,0.6)$ and $(0,0,1)$ are in $\credalset_1\cap\credalset_2$,
and that
$(0.2,0.2,0.6)$ achieves the upper bound for $\{1\}$ and $\{2\}$,
and $(0,0,1)$ achieves the upper bound for $\{3\}$.
We have thereby shown that
$\upr(\{x\})
=\max_{Q\in\credalset_1\cap\credalset_2}Q(\{x\})$ for all $x\in\pspace$.
\end{example}

The following graph summarises the implications between conditions established in this
section:

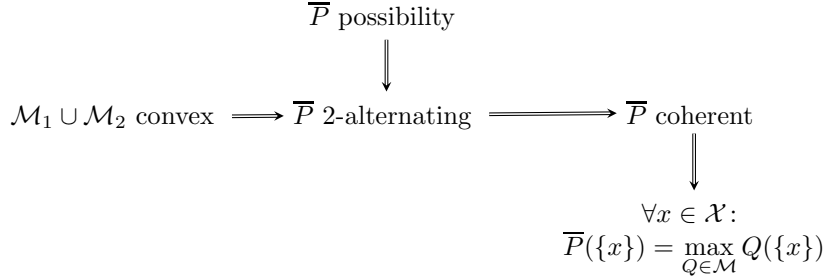
\begin{tikzpicture}
  \matrix (m) [matrix of math nodes,row sep=2em,column sep=2em,minimum width=2em]
  {
      & \upr \text{ possibility } & \\
     \credalset_1\cup\credalset_2 \text{ convex }
     & \upr \text{ $2$-alternating }
     &   \upr \text{ coherent }\\
     & &
     {%
     \begin{array}{c}
     \forall x\in\pspace\colon \\
     \upr(\{x\})=\max\limits_{Q\in\credalset} Q(\{x\})
     \end{array}} \\
     };
  \path[-stealth]
    (m-1-2) edge [double] (m-2-2)
    (m-2-1) edge [double] (m-2-2)
    (m-2-2) edge [double] (m-2-3)
    (m-2-3) edge [double] (m-3-3);
\end{tikzpicture}

The examples in this section show that the converses of these
implications do not hold in general. To see that there is no
implication between the convexity of $\credalset_1\cup\credalset_2$
and $\upr$ being a possibility measure, consider
Example~\ref{ex:convex-not-possib} above as well as
Example~\ref{ex:poss-not-ordered} later on.

\section{When is $\upr$ a possibility measure?}
\label{sec:upr:possib}

Next, we are going to study under which conditions the conjunction
$\upr$
of two possibility measures is again a possibility measure. We shall
begin by providing a simple sufficient (yet not necessary)
condition,
followed by more advanced necessary and sufficient conditions.
One of them will establish a link with game theory,
along with a corresponding method for graphical verification.

\subsection{Sufficient conditions} \label{sec:sufficient-cond}

Clearly, $\upr$ is a possibility measure (and
therefore also coherent) when $\pi_1(x)\leq\pi_2(x)$ for all
$x\in\pspace$, or equivalently, when $\Pi_1(A)\leq\Pi_2(A)$ for all
$A\in\pspace$, since then
$\credalset(\upr)=\credalset_1\cap\credalset_2=\credalset_1$. This
condition means that the possibility measure $\Pi_1$ is more
\emph{specific} \cite{1983:yager,1986:dubois} than $\Pi_2$. However,
this is not the only case in which the conjunction of possibility
measures is again a possibility measure, as the following example
shows.

\begin{example}\label{ex:poss-not-ordered}
Let $\pspace=\{1,2,3\}$ and
\begin{center}
\begin{tabular}{c|ccc}
   & 1 & 2 & 3 \\\hline
  $\pi_1$ & 1 & 0.5 & 0.7 \\
  $\pi_2$ & 1 & 0.6 & 0.6
\end{tabular}
\end{center}
Then
\begin{align}
 &\upr(\{1\})=1,\ \upr(\{2\})=0.5,\ \upr(\{3\})=0.6 \\
 &\upr(\{1,2\})=1,\ \upr(\{1,3\})=1,\ \upr(\{2,3\})=0.6.
\end{align}
Thus, $\upr$ is a possibility measure, even though
$\pi_1(2)<\pi_2(2)$ and $\pi_1(3)>\pi_2(3)$. We can also note that,
in this case, $\credalset_1\cup\credalset_2$ is not convex:
$\Pi_1(\{2\})<\Pi_2(\{2\})$, $\Pi_1(\{3\})>\Pi_2(\{3\})$, and yet
$\Pi_1(\{2,3\})=0.7\neq 0.6=\Pi_2(\{2,3\})$; now use
Corollary~\ref{cor:convex-is-pretty-strong}.
\end{example}

In the example, the possibility distributions $\pi_1$ and $\pi_2$
follow the same order, in the sense that
$\pi_i(2)\leq\pi_i(3)\leq\pi_i(1)$ for both $i=1$ and $i=2$. This
ordering condition turns out to be sufficient for the conjunction of
the two possibility measures to be again a possibility measure:

\begin{theorem}\label{th:order}
$\upr$ is a possibility measure
whenever
there is an ordering $x_1$, \dots, $x_n$
of the elements of $\pspace$ such that
for both $i=1$ and $i=2$ we have that
\begin{equation}\label{eq:thm-ordering}
  \pi_i(x_1)\leq\pi_i(x_2)\leq\dots\leq\pi_i(x_n).
\end{equation}
\end{theorem}

\begin{proof}
Consider $A\subseteq\pspace$ and let
$j(A)\coloneqq\max\{j\in\{1,\dots,n\}\colon x_j\in A\}$.
Then, by Eq.~\eqref{eq:thm-ordering}, $\Pi_i(A)=\pi_i(x_{j(A)})$, and so
\begin{align}
\upr(A)
&=\min\{\Pi_1(A),\Pi_2(A)\}
=\min\{\pi_1(x_{j(A)}),\pi_2(x_{j(A)})\}
\\
&=\upr(\{x_{j(A)}\})
=\max_{x_i\in A}\upr(\{x_i\})
\end{align}
where the last equality follows from
\begin{equation}
  \upr(\{x_1\})\leq\upr(\{x_2\})\leq\dots\leq\upr(\{x_n\}),
\end{equation}
which also follows from Eq.~\eqref{eq:thm-ordering}.
Thus,
$\upr$ is a possibility measure.
\end{proof}

Equivalently, this means that $\upr$ is a possibility measure when
$\pi_1$ and $\pi_2$ are comonotone functions. To see that this
sufficient condition is not necessary, simply note that it may not
hold when $\pi_1\leq\pi_2$:

\begin{example}
Let $\pspace=\{1,2,3\}$ and
\begin{center}
\begin{tabular}{c|ccc}
   & 1 & 2 & 3 \\\hline
  $\pi_1$ & 1 & 0.9 & 0.8 \\
  $\pi_2$ & 1 & 0.5 & 0.6
\end{tabular}
\end{center}
Then $\Pi_2\leq\Pi_1$, so $\upr=\min\{\Pi_1,\Pi_2\}=\Pi_2$ is a
possibility measure. However, $\pi_1$ and $\pi_2$ are not comonotone
because
$\pi_1(2)>\pi_1(3)$ and $\pi_2(2)<\pi_2(3)$.
\end{example}

\subsection{Sufficient and necessary conditions}

Next we give a necessary and sufficient condition for $\upr$
to be a possibility measure. It will allow us to make a link with
game theory.

\begin{theorem}\label{thm:minmax}
  $\upr$ is a possibility measure $\Pi$ if and only if
  \begin{equation}\label{eq:thm:minmax}
    \min\left\{\max_{x\in A}\pi_1(x),\max_{x\in A}\pi_2(x)\right\}
    =
    \max_{x\in A}\min\{\pi_1(x),\pi_2(x)\}
  \end{equation}
  for all non-empty $A\subseteq\pspace$. In such a case, $\unex$ coincides with
  $\upr$, and whence, $\unex$ is a possibility measure as well.
\end{theorem}
\begin{proof}
  Note that the left hand side is $\upr(A)$.

  ``if''. If the equality holds, then $\upr$ is a possibility measure,
  and therefore is coherent. Whence, $\unex=\upr$,
  and so $\unex$ is a possibility measure too.

  ``only if''. On the one hand, by the definition of $\upr$,
  \begin{equation}
    \upr(A)
    =\min\{\Pi_1(A),\Pi_2(A)\}
    =\min\left\{\max_{x\in A}\pi_1(x),\max_{x\in A}\pi_2(x)\right\}
  \end{equation}
  On the other hand, if $\upr$ is a possibility measure, its
  possibility distribution must be
  $\pi(x)=\upr(\{x\})=\min\{\pi_1(x),\pi_2(x)\}$, and so,
  \begin{equation}
    \upr(A)
    =\max_{x\in A}\min\{\pi_1(x),\pi_2(x)\}.
  \end{equation}
  Combining both equalities, we arrive at the desired equality.
\end{proof}

Theorem~\ref{thm:minmax} has a very nice game-theoretic
interpretation. Consider a zero-sum game with two players, where
player 1 can choose $\alpha$ from $\{1,2\}$ and player 2 can choose
$\beta$ from $\{1,\dots,n\}$, with the following payoffs to player
1:
\begin{center}
  \begin{tabular}{c|ccc}
    & $\beta=1$ & \dots & $\beta=n$ \\
    \hline
    $\alpha=1$ & $a_{11}$ & \dots & $a_{1n}$ \\
    $\alpha=2$ & $a_{21}$ & \dots & $a_{2n}$
  \end{tabular}
\end{center}
This table with payoffs to player 1 is called the \emph{payoff
matrix}. For example, if $(\alpha,\beta)=(2,3)$, then player 1 gains
$a_{23}$ and player 2 loses $a_{23}$. A pair $(\alpha,\beta)$ is
called \emph{pure strategy}.

A pure strategy $(\hat{\alpha},\hat{\beta})$ is said to be in
\emph{equilibrium} if it does not benefit either player to change
his choice if the other does not change his choice
\cite[p.~62--64]{1957:luce:raiffa}:
\begin{equation}
  a_{\hat{\alpha}\hat{\beta}}=\max_{\alpha}a_{\alpha \hat{\beta}}=\min_{\beta}a_{\hat{\alpha}\beta}
\end{equation}

\begin{theorem}
  \label{thm:minmax:game}
  $\upr$ is a possibility measure $\Pi$ if and only if
  for all non-empty $A\subseteq\pspace$,
  the zero-sum game with
  choices $\alpha\in\{1,2\}$ and $\beta\in A$,
  and payoffs $a_{\alpha\beta}\coloneqq-\pi_{\alpha}(\beta)$,
  has a pure equilibrium strategy.
\end{theorem}
\begin{proof}
  ``if''. If the zero-sum game associated with $A\subseteq\pspace$
  has a pure equilibrium strategy
  $(\hat{\alpha},\hat{\beta})$,
  then
  \cite[p.~67]{1957:luce:raiffa}
  \begin{equation}
    a_{\hat{\alpha}\hat{\beta}}
    =\max_{\alpha}\min_{\beta}a_{\alpha \beta}=\min_{\beta}\max_{\alpha}a_{\alpha\beta}.
  \end{equation}
  But $a_{\alpha\beta}\coloneqq-\pi_{\alpha}(\beta)$,
  so this is precisely Equation~\eqref{eq:thm:minmax}.

  ``only if''. If $\upr$ is a possibility measure,
  then Equation~\eqref{eq:thm:minmax} can be rewritten as
  \begin{equation}
    \max_{\alpha}\min_{\beta}a_{\alpha \beta}=\min_{\beta}\max_{\alpha}a_{\alpha\beta}.
  \end{equation}
  This means that the zero-game has a pure equilibrium strategy,
  for example
  \begin{align}
    \hat{\alpha}&\coloneqq\argmax_{\alpha}\min_{\beta}a_{\alpha \beta}
    &
    \hat{\beta}&\coloneqq\argmin_{\beta}a_{\hat{\alpha} \beta}
  \end{align}
\end{proof}

Although Theorem~\ref{thm:minmax:game} is in essence nothing more
but a rephrasing of Theorem~\ref{thm:minmax}, it highlights an
interesting fact: we can use any method for solving $2\times n$
zero-sum games in order to determine whether our conjunction $\upr$
is a possibility measure.

The traditional way of finding pure equilibrium strategies goes by
removing dominated options from the game, until only a single
strategy remains. For $2\times n$ games, this is a particularly
simple process: it suffices first to remove columns that are not
optimal for player 2, and then to check whether, in the payoff
matrix that remains, one of the rows dominates the other. For
example, consider the following $2\times 4$ game with the following
payoff to player 1:
\begin{center}
  \begin{tabular}{c|cccc}
    & $\beta=1$ & $\beta=2$ & $\beta=3$ & $\beta=4$ \\
    \hline
    $\alpha=1$ & 3 & 2 & 2 & 4 \\
    $\alpha=2$ & 0 & 3 & 1 & 0
  \end{tabular}
\end{center}
We can remove the column $\beta=2$ because its payoff is higher than
the payoff of column $\beta=3$ regardless of $\alpha$---%
remember that the column player wants to minimize the payoff. We can
also remove the column $\beta=4$ because its payoff is higher than
the payoff of column $\beta=1$ regardless of $\alpha$. No further
columns can be removed. Now, in the remaining payoff matrix, the row
$\alpha=2$ can be removed
because its payoff is lower than the payoff of row $\alpha=1$---%
remember that the row player wants to maximize the payoff. So, the
row player will play $\alpha=1$. In the remaining row $\alpha=1$,
clearly $\beta=3$ achieves the minimum payoff for player 2. This
game therefore has a pure equilibrium strategy, namely
$(\hat{\alpha},\hat{\beta})=(1,3)$.

The two sufficient conditions provided in Section~\ref{sec:sufficient-cond}
follow immediately from Theorem~\ref{thm:minmax:game}.
Indeed,
let $A=\{a_1,a_2,\dots,a_m\}\subseteq\pspace$.
By Theorem~\ref{thm:minmax:game}, we need to consider the payoff matrix
\begin{center}
  \begin{tabular}{c|ccccc}
    & $\beta=1$ & \dots &  $\beta=k$ & \dots & $\beta=m$ \\
    \hline
    $\alpha=1$ & $-\pi_1(a_1)$ & \dots &  $-\pi_1(a_k)$ & \dots & $-\pi_1(a_m)$ \\
    $\alpha=2$ & $-\pi_2(a_1)$ & \dots &  $-\pi_2(a_k)$ & \dots & $-\pi_2(a_m)$
  \end{tabular}
\end{center}
\begin{itemize}
\item
If $\pi_1(x) \leq \pi_2(x)$ for all $x\in\pspace$, then clearly
$-\pi_1(x) \ge -\pi_2(x)$ for every $x \in A$, regardless of $A$.
Therefore the first row of the payoff matrix will dominate the
second row. As player 1 aims to maximize his payoff, $\alpha=1$ will
achieve his optimal strategy, regardless of what player 2 does.
Consequently, the second row can be eliminated, and the pure
equilibrium is reached for $\alpha=1$ and
$\beta=\argmin_{k\in\{1,\dots,m\}}\{-\pi_1(a_k)$\}.
\item
If there is an ordering $x_1$, \dots, $x_n$
of the elements of $\pspace$ such that
$\pi_i(x_j) \leq \pi_i(x_{j+1})$ for all
$i\in\{1,2\}$ and $j \in \{1,\dots,n-1\}$
then, without loss of generality,
we may assume that the elements $a_1$, \dots, $a_m$ of $A$
are ordered reversely, that is,
$-\pi_i(a_k) \leq -\pi_i(a_{k+1})$ for all
$i\in\{1,2\}$ and $k \in \{1,\dots,m-1\}$.
But then the first column
is dominated by all other columns.
As player 2 aims to minimize his payoff,
$\beta=1$ will achieve his optimal strategy,
regardless of what player 1 does.
Consequently, all columns other than the first can be eliminated,
and the pure equilibrium strategy is reached for
$\alpha=\argmax_{i\in\{1,2\}}\{-\pi_i(a_1)\}$
and $\beta=1$.
\end{itemize}

It is important to note that not every $2\times n$ game has a pure
equilibrium. For example, consider the $2\times 2$ zero-sum game
with the following payoff matrix:
\begin{center}
  \begin{tabular}{c|cccc}
    & $\beta=1$ & $\beta=2$ \\
    \hline
    $\alpha=1$ & 1 & 0 \\
    $\alpha=2$ & 0 & 1
  \end{tabular}
\end{center}

Luce and Raiffa \cite[Appendices~3 and~4]{1957:luce:raiffa} discuss
two very nice graphical ways of representing and solving $2\times n$
zero-sum games. Both methods are particularly suited also to
determine whether there are pure equilibrium points. Without going
into too much detail, their first method makes it easy to identify
whether player 1 has a pure equilibrium strategy, whilst their
second method makes it easy to identify whether player 2 has a pure
equilibrium strategy. Because player 2 must have a pure equilibrium
strategy whenever player 1 has a pure one, the first method is most
straightforward for our purpose.

First, we draw all lines $f_\beta(p)\coloneqq
pa_{1\beta}+(1-p)a_{2\beta}$, for $p\in[0,1]$ and all $\beta\in A$.
We then determine the lower envelope $f_A(p)$ of these lines:
\begin{equation}
  f_A(p)\coloneqq\min_{\beta\in A}f_\beta(p).
\end{equation}
Note that $f_A$ will be a concave function. If $f_A$ is monotone
(i.e. has its maximum at $p=0$ or $p=1$), then there is a pure
equilibrium point.

A further substantial gain can be made by recognising that the
monotonicity of a concave function $f_A(p)$ between $p=0$ and $p=1$
is uniquely determined by $f_A'(0)$ and $f_A'(1)$: $f$ is monotone
if and only if $f_A'(0)f_A'(1)\ge 0$. Because $f_A(p)$ is piece-wise
linear, it suffices therefore to look at the left-most line and
right-most line only: the lower envelope is monotone if and only if
these lines are sloped in the same direction. Consequently, for
application to Theorem~\ref{thm:minmax:game}, it suffices to look at
pairs of lines.
In fact, it suffices to look at pairs of intersecting lines,
because if the lines do not intersect,
then the lower envelope is linear and so guaranteed to be monotone.

We have thus reached the following rather surprising result,
for which we also give a simple proof that does not rely
on zero-sum games:

\begin{theorem}\label{thm:minmax:2}
  $\upr$ is a possibility measure if and only if
  \begin{equation}\label{eq:thm:minmax:2}
    \min_{i\in\{1,2\}}\left(\max_{j\in\{1,2\}}\pi_i(x_j)\right)
    =
    \max_{j\in\{1,2\}}\left(\min_{i\in\{1,2\}}\pi_i(x_j)\right)
  \end{equation}
  for all $\{x_1,x_2\}\subseteq\pspace$.
  In such a case, $\unex$ coincides with
  $\upr$, and whence, $\unex$ is a possibility measure as well.
\end{theorem}

\begin{proof}
First, note that Eq.~\eqref{eq:thm:minmax:2} is equivalent to
saying that
\begin{equation}\label{eq:thm:minmax:2:alt}
  \upr(\{x_1,x_2\})=\max\{\upr(\{x_1\}),\upr(\{x_2\})\}
\end{equation}
for every $\{x_1,x_2\}\subseteq\pspace$.
We show that this is indeed equivalent
to $\upr$ being a possibility measure

`if'.
Consider any non-empty $A\subseteq\pspace$.
Let
\begin{align}
x_1&\coloneqq\argmax_{x\in A}\pi_1(x),
&
x_2&\coloneqq\argmax_{x\in A}\pi_2(x).
\end{align}
Since $\Pi_1$ and $\Pi_2$ are possibility measures,
it immediately follows that
\begin{align}
\Pi_1(A)=\Pi_1(\{x_1\})&=\Pi_1(\{x_1,x_2\}),
\\
\Pi_2(A)=\Pi_2(\{x_2\})&=\Pi_2(\{x_1,x_2\}).
\end{align}
Consequently,
\begin{align}
\upr(A)
&=\min\{\Pi_1(A),\Pi_2(A)\}
\\
&=\min\{\Pi_1(\{x_1,x_2\}),\Pi_2(\{x_1,x_2\})\}
\\
&=\upr(\{x_1,x_2\})
\\
\intertext{and now applying Eq.~\eqref{eq:thm:minmax:2},}
&=\max\{\upr(\{x_1\}),\upr(\{x_2\})
\\
&\le\max_{x\in A}\upr(\{x\})
\end{align}

The converse inequality follows by
monotonicity of $\upr$---indeed,
both $\Pi_1$ and $\Pi_2$ are monotone,
so their minimum must be monotone too.
Specifically, for every $x\in A$ we have that
$\Pi_1(A)\ge\pi_1(x)$ and $\Pi_2(A)\ge\pi_2(x)$, so
\begin{align}
\upr(A)
=\min\{\Pi_1(A),\Pi_2(A)\}
\ge\min\{\pi_1(x),\pi_2(x)\}
=\upr(\{x\})
\end{align}
and therefore $\upr(A)\ge\max_{x\in A}\upr(\{x\})$. Thus,
$\upr(A)=\max_{x\in A}\upr(\{x\})$ and as a consequence it is a
possibility measure.

`only if'.
If $\upr$ is a  possibility measure
then $\upr(A)=\max_{x\in A}\upr(\{x\})$ for all
non-empty $A\subseteq\pspace$,
and in particular also for all $A=\{x_1,x_2\}$.
Eq.~\eqref{eq:thm:minmax:2:alt} follows.
\end{proof}

\subsection{Examples}
\label{sec:graphical-method-examples}

The verification of Theorem~\ref{thm:minmax:game} entails looking at
every pair of lines $f_{\beta}$ and $f_{\gamma}$, and checking:
\begin{itemize}
\item whether $f_{\beta}$ and $f_{\gamma}$ intersect for some $0<p<1$,
that is, whether $f_{\beta}(p)=f_{\gamma}(p)$ for some $0<p<1$;
\item if so, whether $f_{\beta}$ and $f_{\gamma}$ have the same slope.
\end{itemize}
If for all intersecting pairs, both lines have the same slope,
then the conditions of Theorem~\ref{thm:minmax:game} are satisfied,
and the conjunction will be a possibility measure.

Let us first provide an example, inspired by Sandri \textit{et al.}~\cite{1995:sandri},
where the conditions hold.

\begin{example}\label{exm:hold}
Two economists provide their opinion
about the value ($\mathcal{X}=\{1,\ldots,9\}$) of a future stock market:
\begin{center}
\begin{tabular}{c|cccccccccc}
   & 1  & 2 & 3 & 4 & 5 & 6 & 7 & 8 & 9 \\\hline
  $\pi_1$ & 1 & 0.95 & 0.95 & 0.8 & 0.7 & 0.2 & 0.3 & 0.1 & 0.05 \\
  $\pi_2$ & 1 & 0.8 & 0.6 & 0.7 & 0.6 & 0.6 & 0.3 & 0.4 & 0.1 \end{tabular}
\end{center}
which are pictured as $f_{\beta}$ for $\beta \in \{1,\ldots,9\}$ in
Figure~\ref{fig:exampleok}. We actually pictured $-f_{\beta}$, to
make it easier to relate the lines to the possibility distributions.
It can be checked that the conditions required by
Theorem~\ref{thm:minmax:game} hold for every pair. This means that
the merged opinion $\upr$ of the two economists can be represented
as a possibility distribution. Figure~\ref{fig:exampleok} makes
verification even easier: there are only three intersecting pairs,
namely $(f_3,f_4)$, $(f_6,f_7)$, and $(f_7,f_8)$, and in each pair,
both lines have the same slope. Consequently, $\upr$ is a
possibility measure induced by the possibility distribution
\begin{center}
\begin{tabular}{c|ccccccccc}
   & 1 & 2 & 3 & 4 & 5 & 6 & 7 & 8 & 9  \\\hline
  $\pi$ & 1 & 0.8 & 0.6 & 0.7 & 0.6 & 0.2 & 0.3 & 0.1 & 0.05
\end{tabular}
\end{center}
\end{example}

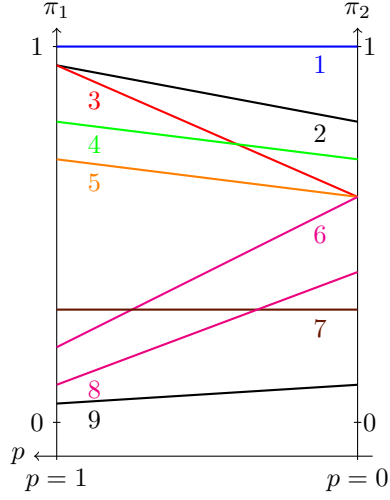
\begin{figure}
\begin{center}
\begin{tikzpicture}
\draw[->] (0,-0.5)  -- (0,5.25) node[above] {$\pi_1$};
\draw (0.05,0) --  (-0.05,0) node[left] {0};
\draw (0.05,5) --  (-0.05,5) node[left] {1};
\draw[->] (4,-0.5)  -- (4,5.25) node[above] {$\pi_2$};
\draw (4.05,0) --  (3.95,0) node[right] {0};
\draw (4.05,5) --  (3.95,5) node[right] {1};
\draw[blue,thick] (0,1*5) -- (4,1*5) node[very near end,below] {$1$};
\draw[thick] (0,0.95*5) -- (4,0.8*5) node[very near end, below] {$2$};
\draw[red,thick] (0,0.95*5) -- (4,0.6*5) node[very near start,below] {$3$};
\draw[green,thick] (0,0.8*5) -- (4,0.7*5) node[very near start,below] {$4$};
\draw[orange,thick] (0,0.7*5) -- (4,0.6*5) node[very near start,below] {$5$};
\draw[magenta,thick] (0,0.2*5) -- (4,0.6*5) node[very near end,below] {$6$};
\draw[Sepia,thick] (0,0.3*5) -- (4,0.3*5) node[very near end,below] {$7$};
\draw[RubineRed,thick] (0,0.1*5) -- (4,0.4*5) node[very near start,below] {$8$};
\draw[thick] (0,0.05*5) -- (4,0.1*5) node[very near start,below] {$9$};
\draw[->] (4.2,-0.45) -- (-0.3,-0.45) node[left] {$p$};
\node[below] at (0,-0.5) {$p=1$};
\node[below] at (4,-0.5) {$p=0$};
\end{tikzpicture}
\end{center}
\caption{Example~\ref{exm:hold} game-theoretic figure. }
\label{fig:exampleok}
\end{figure}

When $\pi_1$ and $\pi_2$ do not satisfy
the conditions of Theorem~\ref{thm:minmax:game},
our graphical verification technique
also allows us to heuristically adjust $\pi_1$ and $\pi_2$
into new possibility distributions that do satisfy
the conditions of Theorem~\ref{thm:minmax:game}.
The next example illustrates this heuristic procedure.

\begin{example}\label{exm:donothold}
Two economists provide the following opinions:
 \begin{center}
 \begin{tabular}{c|cccccccc}
    & 1 & 2 & 3 & 4 & 5 & 6 & 7 & 8 \\\hline
   $\pi_1$ & 1 & 0.9 & 0.7 & 0.6 & 0.5 & 0.4 & 0.3 & 0.1\\
   $\pi_2$ & 0.8 &  0.2 & 1 & 0.6 & 0.1 & 0.2 & 0.3 & 0.9
 \end{tabular}
 \end{center}
The left hand side of
Figure~\ref{fig:examplenotok} depicts our graphical method.
Many pairs of intersecting lines have opposite slopes,
for instance $(f_8,f_2)$.
Therefore, the conditions of Theorem~\ref{thm:minmax:game}
are not satisfied.
Interestingly, there is no
$x \in \mathcal{X}$ such that $\pi_1(x)=\pi_2(x)=1$---%
this is a necessary condition for $\upr$ to be a possibility
measure; see proof of
Lemma~\ref{lem:possibonsingletons}\ref{lem:possibonsingletons:uprpossib}
further on.

A possible adjustment that allows to satisfy the conditions of
Theorem~\ref{thm:minmax:game}, can be done for example by modifying
$f_1$, $f_2$ and $f_8$, so that $f_1$ and $f_2$ become positively
slopped,
and so that $f_8$ no longer intersects with $f_5$---%
of course, conservative adjustments
should only be done by moving lines upwards.
The right hand side of Figure~\ref{fig:examplenotok}
shows the adjusted lines dashed.
They result in the following adjusted possibility distributions:
 \begin{center}
 \begin{tabular}{c|cccccccc}
    & 1 & 2 & 3 & 4 & 5 & 6 & 7 & 8 \\\hline
   $\pi'_1$ & 1 & 0.9 & 0.7 & 0.6 & 0.5 & 0.4 & 0.3 & 0.5\\
   $\pi'_2$ & 1 &  0.9 & 1 & 0.6 & 0.1 & 0.2 & 0.3 & 0.9
 \end{tabular}
 \end{center}
The resulting adjusted conjunction is:
 \begin{center}
 \begin{tabular}{c|cccccccc}
    & 1 & 2 & 3 & 4 & 5 & 6 & 7 & 8 \\\hline
   $\pi'$ & 1 & 0.9 & 0.7 & 0.6 & 0.1 & 0.2 & 0.3 & 0.5
 \end{tabular}
 \end{center}
\end{example}

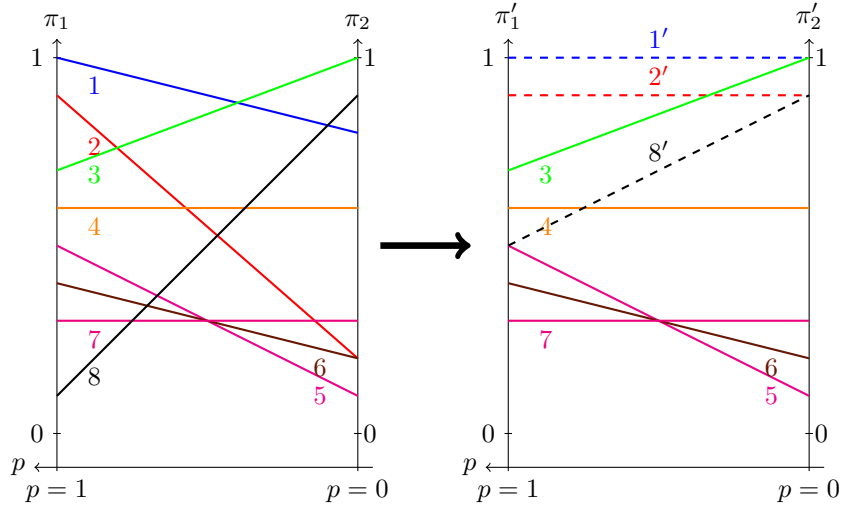
\begin{figure}
 \begin{center}
 \begin{tikzpicture}
 \draw[->] (0,-0.5)  -- (0,5.25) node[above] {$\pi_1$};
 \draw (0.05,0) --  (-0.05,0) node[left] {0};
 \draw (0.05,5) --  (-0.05,5) node[left] {1};
 \draw[->] (4,-0.5)  -- (4,5.25) node[above] {$\pi_2$};
 \draw (4.05,0) --  (3.95,0) node[right] {0};
 \draw (4.05,5) --  (3.95,5) node[right] {1};
 \draw[blue,thick] (0,1*5) -- (4,0.8*5) node[very near start,below] {$1$};
 \draw[red,thick] (0,0.9*5) -- (4,0.2*5) node[very near start,below] {$2$};
 \draw[green,thick] (0,0.7*5) -- (4,1*5) node[very near start,below] {$3$};
 \draw[orange,thick] (0,0.6*5) -- (4,0.6*5) node[very near start,below] {$4$};
 \draw[magenta,thick] (0,0.5*5) -- (4,0.1*5) node[very near end,below] {$5$};
 \draw[Sepia,thick] (0,0.4*5) -- (4,0.2*5) node[very near end,below] {$6$};
 \draw[RubineRed,thick] (0,0.3*5) -- (4,0.3*5) node[very near start,below] {$7$};
 \draw[black,thick] (0,0.1*5) -- (4,0.9*5) node[very near start,below] {$8$};
 \draw[->] (4.2,-0.45) -- (-0.3,-0.45) node[left] {$p$};
 \node[below] at (0,-0.5) {$p=1$};
 \node[below] at (4,-0.5) {$p=0$};

  \draw[line width=2.5pt,->] (4.3,2.5) -- (5.5,2.5);

  \begin{scope}[xshift=6cm]
  \draw[->] (0,-0.5)  -- (0,5.25) node[above] {$\pi'_1$};
   \draw (0.05,0) --  (-0.05,0) node[left] {0};
   \draw (0.05,5) --  (-0.05,5) node[left] {1};
   \draw[->] (4,-0.5)  -- (4,5.25) node[above] {$\pi'_2$};
   \draw (4.05,0) --  (3.95,0) node[right] {0};
   \draw (4.05,5) --  (3.95,5) node[right] {1};
    \draw[blue,thick,dashed] (0,1*5) -- (4,1*5) node[midway,above] {$1'$};
    \draw[red,thick,dashed] (0,0.9*5) -- (4,0.9*5) node[midway,above] {$2'$};
   \draw[green,thick] (0,0.7*5) -- (4,1*5) node[very near start,below] {$3$};
   \draw[orange,thick] (0,0.6*5) -- (4,0.6*5) node[very near start,below] {$4$};
   \draw[magenta,thick] (0,0.5*5) -- (4,0.1*5) node[very near end,below] {$5$};
   \draw[Sepia,thick] (0,0.4*5) -- (4,0.2*5) node[very near end,below] {$6$};
   \draw[RubineRed,thick] (0,0.3*5) -- (4,0.3*5) node[very near start,below] {$7$};
    \draw[black,thick,dashed] (0,0.5*5) -- (4,0.9*5) node[midway,above] {$8'$};
    \draw[->] (4.2,-0.45) -- (-0.3,-0.45) node[left] {$p$};
    \node[below] at (0,-0.5) {$p=1$};
    \node[below] at (4,-0.5) {$p=0$};
  \end{scope}
 \end{tikzpicture}
 \end{center}
 \caption{Example~\ref{exm:donothold} game-theoretic figure and adjustment (in dashed). }
 \label{fig:examplenotok}
\end{figure}

It is clear that any upward adjustment implies a loss of information.
In general, there is no unique adjustment minimizing this loss.
In any case, upward adjustment ensures that the obtained result will be
consistent with the initial information,
as it will give an outer approximation.

If there is an element $x$ such that $\pi_1(x)=\pi_2(x)=1$,
then adjustments can also be done downwards,
in which case the obtained approximation would be an inner approximation.

\section{When is $\unex$ a possibility measure?}
\label{sec:unex:possib}

The above condition for $\upr$ to be a possibility measure is
obviously sufficient for $\unex$ to be a possibility measure.
However, the condition is not necessary, as shown by the next
example:

\begin{example}\label{ex:unex-upr-not-equiv}
Let
\begin{center}
  \begin{tabular}{c|ccc}
    & 1 & 2 & 3 \\
    \hline
    $\pi_1$ & 1 & 1 & 0 \\
    $\pi_2$ & 1 & 0 & 1
  \end{tabular}
\end{center}
The credal set of the conjunction is the singleton
$\credalset=\{\apr\}$ for which $\apr(\{1\})=1$ (and zero
elsewhere), because this is the only probability measure that
satisfies $\apr(\{x\})\le\upr(\{x\})$ for all $x$. Whence, the
natural extension $\unex$ of $\upr$ is obviously a possibility
measure.

Nevertheless, $\upr$ is \emph{not} a possibility measure. Indeed,
\begin{equation}
  \min_{i\in\{1,2\}}\left(\max_{j\in\{2,3\}}\pi_i(j)\right)
  =
  \max_{j\in\{2,3\}}\left(\min_{i\in\{1,2\}}\pi_i(j)\right)
\end{equation}
as the left hand side is one, and the right hand side is zero. This
is because $\upr$ is not a coherent upper probability, since
$\upr(\{2,3\})=1>\upr(\{2\})+\upr(\{3\})$.
\end{example}

Indeed, when $\upr$ is coherent then it coincides with $\unex$, and
therefore in that case $\unex$ is a possibility measure if and only
if $\upr$ is. Below, we state a number of necessary conditions for
$\unex$ to be a possibility measure. So far, we failed to identify a
condition that is both sufficient \emph{and} necessary.

\begin{lemma}
  \label{lem:pixone}
  If $\unex$ is a possibility measure $\Pi$, then there is an
  $x\in\pspace$ such that $\pi(x)=\pi_1(x)=\pi_2(x)=1$.
\end{lemma}
\begin{proof}
  If $\unex$ is a possibility measure, then $\unex(\{x\})=\pi(x)=1$ for at
  least one $x\in\pspace$. For any such $x$,
  \begin{equation}
    1=\unex(\{x\})\le\min\{\pi_1(x),\pi_2(x)\},
  \end{equation}
  whence, it can only be that $\pi_1(x)=\pi_2(x)=1$ for such $x$.
\end{proof}

Of course, if $\unex$ is a possibility measure and $\upr$ is not
coherent, then $\unex$ and $\upr$ will not coincide on all events.
We shall show next that they are always guaranteed to coincide on
the singletons. In order to see this, note that if $\upr$ is a
possibility measure, then its possibility distribution is given by
\begin{equation}\label{eq:possib-distrib-1}
  \pi(x)\coloneqq\min\{\pi_1(x),\pi_2(x)\}=\upr(\{x\}).
\end{equation}
We denote the possibility measure determined by this distribution by
\begin{equation}\label{eq:possib-distrib-2}
  \Pi(A)\coloneqq\max_{x\in A}\pi(x).
\end{equation}
We can establish the following.

\begin{lemma}\label{lem:possibonsingletons}
The following statements hold.
\begin{enumerate}[(a)]
 \item\label{lem:possibonsingletons:uprgepi}
 $\upr\geq\Pi$.
 \item\label{lem:possibonsingletons:normed}
 $\Pi$ is normed if and only if
 there is an $x\in\pspace$ such that $\upr(\{x\})=1$.
 In that case, $\upr$ avoids sure loss and
 $\upr\geq\unex\geq\Pi$.
 \item\label{lem:possibonsingletons:uprpossib}
 $\upr$ is a possibility measure if and only if $\upr=\Pi$.
 \item\label{lem:possibonsingletons:unexpossib}
 $\unex$ is a possibility measure if and only if $\unex=\Pi$.
\end{enumerate}
\end{lemma}

\begin{proof}
\begin{enumerate}[(a)]
\item Consider any $A\subseteq\pspace$. Observe that, for any $x\in A$,
  \begin{align}
    \max_{x'\in A}\pi_1(x')&\ge\pi_1(x)\ge\pi(x), \\
    \max_{x'\in A}\pi_2(x')&\ge\pi_2(x)\ge\pi(x).
  \end{align}
  Whence,
  \begin{equation}
    \upr(A)=\min\left\{\max_{x'\in A}\pi_1(x'),\max_{x'\in A}\pi_2(x')\right\}
    \ge
    \pi(x)
  \end{equation}
  for all $x\in A$. We immediately arrive at the desired inequality.

\item $\Pi$ is normed if and only if there is some $x\in\pspace$
such that $\pi(x)=\upr(\{x\})=1$. In that case, the degenerate
probability measure on $x$ belongs to
$\credalset_1\cap\credalset_2$, and as a consequence $\upr$ avoids
sure loss. Moreover, $\Pi$ is then a coherent upper probability that
is dominated by $\upr$, whence $\Pi$ must also be dominated by the natural
extension $\unex$ of $\upr$,
because $\unex$ is the point-wise largest coherent upper probability
that is dominated by $\upr$ \cite[3.1.2(e)]{1991:walley}.

\item If $\upr$ is a possibility measure,
then $\upr(\{x\})=1$ for some $x\in\pspace$.
Consequently, by \ref{lem:possibonsingletons:normed}
\begin{equation}\label{eq:lem:pf:possibonsingletons:1}
  \upr(A)\geq\unex(A)\geq\Pi(A)\text{ for all }A\subseteq\pspace.
\end{equation}
Because $\upr(\{x\})=\min\{\pi_1(x),\pi_2(x)\}=\Pi(\{x\})$
for all $x\in\pspace$,
it follows that also
\begin{equation}\label{eq:lem:pf:possibonsingletons:2}
  \upr(\{x\})=\unex(\{x\})=\Pi(\{x\})\text{ for all }x\in\pspace.
\end{equation}
Because both $\upr$ and $\Pi$ are possibility measures,
they are uniquely determined by their restriction to singletons,
and therefore $\upr=\Pi$.
The converse implication is trivial.

\item Similarly, if $\unex$ is a possibility measure,
then $\unex(\{x\})=1$ for some $x\in\pspace$.
Because $\upr\ge\unex$,
it can only be that also $\upr(\{x\})=1$ for that same $x$.
Consequently, by \ref{lem:possibonsingletons:normed},
Eq.~\eqref{eq:lem:pf:possibonsingletons:1} must hold here as well.
Again, because $\upr(\{x\})=\min\{\pi_1(x),\pi_2(x)\}=\Pi(\{x\})$
for all $x\in\pspace$,
it follows that Eq.~\eqref{eq:lem:pf:possibonsingletons:2}
holds here too.
Because both $\unex$ and $\Pi$ are possibility measures,
they are uniquely determined by their restriction to singletons,
and therefore $\unex=\Pi$.
(Note that $\upr$ does not always coincide with $\Pi$
in this case because $\upr$ may not be a possibility measure;
see Example~\ref{ex:unex-upr-not-equiv}.)
Again, the converse implication is trivial. \hfill $\qedhere$
\end{enumerate}
\end{proof}

To see that $\Pi$ need not be normed for $\upr$ to avoid sure loss
(or even to be coherent), it suffices to consider
Example~\ref{ex:coh-not-possib}.
However, for
$\upr$ to be a possibility measure,
$\Pi$ need to be normed,
as we can deduce from
Lemma~\ref{lem:possibonsingletons}\ref{lem:possibonsingletons:uprpossib}.

Lemma~\ref{lem:possibonsingletons}\ref{lem:possibonsingletons:uprgepi} also indicates that
taking the minimum between two possibility distributions $\pi_1$ and $\pi_2$,
which is the most conservative conjunctive operator in possibility theory,
will always provide an inner approximation of $\upr$ when $\upr$ is not a possibility measure.
In a way, our heuristic method for adjusting possibility distributions
to ensure that the conjunction is a possibility measure
provides an even more conservative conjunctive operator,
which in addition also ensures coherence
unlike the plain minimum operator.

The next result shows that
Example~\ref{ex:unex-upr-not-equiv}
hinges on $\pi_1$ and $\pi_2$ not being strictly positive.

\begin{theorem}\label{thm:notposswhennonzero}
  Let $\pi_1$ and $\pi_2$ be two strictly positive possibility distributions.
  Then $\unex$ is a possibility measure if and only if
  $\upr$ is a possibility measure.
\end{theorem}

\begin{proof}
  `if'.
  If $\upr$ is a possibility measure, then
  $\upr$ is coherent, and therefore coincides with its natural extension.
  So, $\unex$ will be a possibility measure as well.

  `only if'.
If $\unex$ is a possibility measure
then,
by Lemma~\ref{lem:possibonsingletons}\ref{lem:possibonsingletons:unexpossib},
$\unex=\Pi$, with $\pi$ and $\Pi$ defined
as in Eqs.~\eqref{eq:possib-distrib-1} and~\eqref{eq:possib-distrib-2}.
In particular, there is some
$x^*\in\pspace$ such that
$\unex(\{x^*\})=\upr(\{x^*\})=\pi_1(x^*)=\pi_2(x^*)=1$.

Assume ex-absurdo that $\upr$ is not a possibility measure.
By Theorem~\ref{thm:minmax:2},
there must be $\{x_1,x_2\}\subseteq\pspace$ such that
  \begin{equation}
    \min_{i\in\{1,2\}}\left(\max_{j\in\{1,2\}}\pi_i(x_j)\right)
    \neq
    \max_{j\in\{1,2\}}\left(\min_{i\in\{1,2\}}\pi_i(x_j)\right)
  \end{equation}
  This inequality can only hold if the matrix
  \begin{equation}
  \begin{bmatrix}
  \pi_1(x_1) & \pi_1(x_2) \\
  \pi_2(x_1) & \pi_2(x_2)
  \end{bmatrix}
  \end{equation}
  has neither dominating rows nor dominating columns,
  or in other words, we must have either
  \begin{equation}\label{eq:pf:inequalities:cases}
          \begin{matrix}
          \pi_1(x_1) & < & \pi_1(x_2) \\
           \wedge & &  \vee \\
          \pi_2(x_1) & > & \pi_2(x_2)
          \end{matrix}
    \qquad \textrm{or} \qquad
       \begin{matrix}
        \pi_1(x_1) & > & \pi_1(x_2) \\
         \vee & &  \wedge \\
        \pi_2(x_1) & < & \pi_2(x_2)
        \end{matrix}
    \end{equation}
Without loss of generality, we can assume that the first situation
holds, as we can always swap $x_1$ and $x_2$.
From these strict inequalities, it follows that
\begin{align}
\max\{\pi_1(x_1),\pi_2(x_2)\}
=\max_{j\in\{1,2\}}\left(\min_{i\in\{1,2\}}\pi_i(x_j)\right)
=\max\{\unex(\{x_1\}),\unex(\{x_2\})\},
\end{align}
where last equality follows from
Lemma~\ref{lem:possibonsingletons}\ref{lem:possibonsingletons:unexpossib}.
So, if we can show that
\begin{equation}
  \unex(\{x_1,x_2\}) > \max\{\pi_1(x_1),\pi_2(x_2)\},
\end{equation}
then we have established a contradiction.
By Eqs.~\eqref{eq:credalset} and~\eqref{eq:unex},
it suffices to show that
there is a $Q\le\upr$
such that
\begin{equation}
  \label{eq:helper:qgtmax}
  Q(\{x_1,x_2\})>\max\{\pi_1(x_1),\pi_2(x_2)\}.
\end{equation}

Now, a probability measure $Q$ which is zero everywhere
except on $\{x_1,x_2,x^*\}$ satisfies $Q\le\upr$
if and only if all of the following inequalities are satisfied:
\begin{align}
Q(\{x_1\}) &\leq \pi_1(x_1) \label{eq:ineqnoposs} \\
Q(\{x_2\}) &\leq \pi_2(x_2) \label{eq:ineqnoposs2} \\
Q(\{x_1\}) + Q(\{x_2\}) &\leq\min\{\pi_1(x_2),\pi_2(x_1)\}
\label{eq:ineqnoposs3}
\end{align}
Indeed, consider any $A\subseteq\pspace$.
\begin{enumerate}[(a)]
\item If $A\cap\{x_1,x_2,x^*\}=\emptyset$ then $Q(A)=0$,
and no constraints are required.
\item If $x^*\in A\cap\{x_1,x_2,x^*\}$ then $\upr(A)=1$,
and no constraints are required.
\item If $A\cap\{x_1,x_2,x^*\}=\{x_1\}$ then $Q(A)=Q(\{x_1\})$. Clearly,
$Q(\{x_1\})\le\upr(A)$ for all such $A$ if and only if
\begin{equation}
  Q(\{x_1\})\le\upr(\{x_1\})=\min\{\pi_1(x_1),\pi_2(x_1)\}=\pi_1(x_1).
\end{equation}
This is precisely Eq.~\eqref{eq:ineqnoposs}.
\item If $A\cap\{x_1,x_2,x^*\}=\{x_2\}$ then $Q(A)=Q(\{x_2\})$. Clearly,
$Q(\{x_2\})\le\upr(A)$ for all such $A$ if and only if
\begin{equation}
  Q(\{x_2\})\le\upr(\{x_2\})=\min\{\pi_1(x_2),\pi_2(x_2)\}=\pi_2(x_2).
\end{equation}
This is precisely Eq.~\eqref{eq:ineqnoposs2}.
\item If $A\cap\{x_1,x_2,x^*\}=\{x_1,x_2\}$ then we obtain $Q(A)=Q(\{x_1,x_2\})$. Clearly,
$Q(\{x_1,x_2\})\le\upr(A)$ for all such $A$ if and only if
\begin{align}
  Q(\{x_1,x_2\})\le\upr(\{x_1,x_2\})
  &=\min\{\Pi_1(\{x_1,x_2\}),\Pi_2(\{x_1,x_2\})\}
  \\
  &=\min\{\pi_1(x_2),\pi_2(x_1)\}
\end{align}
where the last equality follows from
Eq.~\eqref{eq:pf:inequalities:cases} (left case). This is precisely
Eq.~\eqref{eq:ineqnoposs3}.
\end{enumerate}

So, we are done if we can construct a
probability measure $Q$ on $\{x_1,x_2,x^*\}$
which simultaneously satisfies
Eqs.~\eqref{eq:helper:qgtmax},
\eqref{eq:ineqnoposs},
\eqref{eq:ineqnoposs2},
and~\eqref{eq:ineqnoposs3}.

Also note that we always have $x^*\neq x_1$ and $x^*\neq x_2$ (and
obviously also $x_1\neq x_2$), because
Eq.~\eqref{eq:pf:inequalities:cases} (left case) implies that
$\pi_1(x_1)<1$ and $\pi_2(x_2)<1$, so $\{x_1,x_2,x^*\}$ always
contains exactly three elements.

We consider two cases.

1. If $\pi_1(x_1) + \pi_2(x_2) \leq \min\{\pi_1(x_2),\pi_2(x_1)\}$,
then the probability measure $Q$ with
\begin{align}
  Q(\{x_1\})&\coloneqq \pi_1(x_1), &
  Q(\{x_2\})&\coloneqq \pi_2(x_2), &
  Q(\{x^*\})&\coloneqq 1-(\pi_1(x_1) + \pi_2(x_2))
\end{align}
clearly satisfies Eqs.~\eqref{eq:ineqnoposs},
\eqref{eq:ineqnoposs2},
and~\eqref{eq:ineqnoposs3}.
We also have that
\begin{equation}
Q(\{x_1,x_2\})
=Q(\{x_1\})+Q(\{x_2\})
=\pi_1(x_1)+\pi_2(x_2)
>\max\{\pi_1(x_1),\pi_2(x_2)\}
\end{equation}
because both $\pi_1(x_1)$ and $\pi_2(x_2)$ are strictly positive by
assumption,
so Eq.~\eqref{eq:helper:qgtmax} is satisfied as well,
finishing the proof for this case.

2. If $\pi_1(x_1) + \pi_2(x_2) >
\min\{\pi_1(x_2),\pi_2(x_1)\}$,
then the probability measure $Q$ with
\begin{align}
  Q(\{x_1\})&\coloneqq\pi_1(x_1), \\
  Q(\{x_2\})&\coloneqq\min\{\pi_1(x_2),\pi_2(x_1)\} - \pi_1(x_1), \\
  Q(\{x^*\})&\coloneqq 1- (\min\{\pi_1(x_2),\pi_2(x_1)\} )
\end{align}
clearly satisfies Eqs.~\eqref{eq:ineqnoposs},
\eqref{eq:ineqnoposs2},
and~\eqref{eq:ineqnoposs3}.
We also have that
\begin{equation}
Q(\{x_1,x_2\})=\min\{\pi_1(x_2),\pi_2(x_1)\}>\max\{\pi_1(x_1),\pi_2(x_2)\}
\end{equation}
where the strict inequality follows from
Eq.~\eqref{eq:pf:inequalities:cases} (left case), so
Eq.~\eqref{eq:helper:qgtmax} is satisfied as well, finishing the
proof for this case.
\end{proof}

\section{Exampe: a simple medical diagnosis problem}\label{sec:diagnosis-example}

To conclude this paper, we illustrate our
results on a medical diagnosis problem, inspired by
Palacios \textit{et al.}~\cite{2010:palacios}.

Consider $\pspace=\{d,h,n\}$ where $d$, $h$, and $n$ stand for
dyslexic, hyperactive and no problem, respectively.
As is explained by Palacios \textit{et al.}~\cite{2010:palacios},
it may be difficult for physicians
to recognize between dyslexia and hyperactivity of children,
yet it is important to provide reliable information.

Let us now assume that the available information is expressed by
means of possibility distributions: these may be the result of a
classification process~\cite{2010:palacios} or of an
elicitation procedure. We wish to provide a
joint summary of these distributions which is still
representable as a possibility distribution, for instance
because we want to use it in methods tailored for possibility
distributions, or because it is easier to present possibility
distributions to physicians.

\begin{example}
Two physicians provide the following possibility distributions:
\begin{center}
\begin{tabular}{c|ccc}
& $d$ & $h$ & $n$ \\
\hline
$\pi_1$ & 1 & 0.5 & 0.2 \\
$\pi_2$ & 1 &  0.3 & 0.4
\end{tabular}
\end{center}
The two physicians actually agree that dyslexia is quite possible,
but they are not in agreement on the possibility of the other two options.

The conjunction $\upr\coloneqq\min\{\Pi_1,\Pi_2\}$
avoids sure loss: for example,
the probability measure $Q$ with
$Q(\{d\})=1$ is dominated by $\upr$.
It can be verified that $\upr$ is coherent.
Interestingly,
the condition of Proposition~\ref{pr:charac-coher-prevs}
is not satisfied:
no convex
combination of the probability measures determined by the mass functions
$(0.5,0.3,0.2)\in\credalset_1$ and
$(0.6,0,0.4)\in\credalset_2$ belongs to
$\credalset_1\cup\credalset_2$.

The natural extension $\unex$ of $\upr$, which is the
upper envelope of the credal set $\credalset_1\cap\credalset_2$,
coincides with $\upr$ in this example, because $\upr$
happens to be coherent:
\begin{gather}
\unex(\{d\})=1
\qquad \unex(\{h\})=0.3
\qquad \unex(\{n\})=0.2 \\
\unex(\{h,n\})=0.4
\qquad \unex(\{d,h\})=\unex(\{d,n\})=\unex(\{d,h,n\})=1.
\end{gather}
However, $\unex$ is not a possibility measure
because
\begin{equation}
\unex(\{h,n\})=0.4>\max\{\unex(\{h\},\unex(\{n\})\}=0.3.
\end{equation}

The graphical procedure summarized
at the beginning of Section~\ref{sec:graphical-method-examples}
suggests a possible correction of $\pi_2$
for the conjunction to become a possibility measure:
\begin{center}
\begin{tabular}{c|ccc}
& $d$ & $h$ & $n$ \\ \hline $\pi'_2$ & 1 &  0.4 & 0.4
\end{tabular}
\end{center}
By Theorem~\ref{th:order},
the conjunction of $\pi_1$ and $\pi_2'$ is then a possibility measure
with possibility distribution
\begin{center}
\begin{tabular}{c|ccc}
& $d$ & $h$ & $n$ \\\hline $\pi$ & 1 &  0.4 & 0.2
\end{tabular}
\end{center}
which is still quite informative.
\end{example}

\section{Conclusions}\label{sec:conclusions}

In this paper, we have characterized in different ways the
conjunction of two possibility measures.
In particular, we have addressed the following questions:
\begin{enumerate}
\item When does the conjunction avoid sure loss?
\item When is the conjunction coherent?
\item When is the conjunction again a possibility measure?
\item When is the natural extension of the conjunction
  again a possibility measure?
\end{enumerate}
For each of these, we have provided both sufficient and necessary
conditions. We demonstrated through many examples that these
conditions remain quite restrictive; this seems to be the price to
pay for working with possibility distributions.

From a practical point, one result that we find particularly
interesting is the game-theoretic characterization of the conditions
under which the conjunction is again a possibility measure. Indeed,
this characterization offers a very simple and convenient graphical
verification method. It can also be used in practice to
heuristically adjust possibility distributions to ensure that their
conjunction remains a possibility distribution.

It is not too difficult to extend some of our results to the
conjunction of more than two possibility measures, by noting that
the conjunction can be taken in a pairwise sequential manner. Note
nevertheless that these pairwise conjunctions being possibility
measures is sufficient, but not necessary, for the conjunction of
all the possibility measures to be a possibility measure. For some
other results, such as Theorem~\ref{thm:minmax:2}, some adjustments
should be made.

As for future lines of research, we would like to point out a few.
It would be interesting to study under what conditions
possibility measures are closed under other combination rules, such
as those discussed in \cite{1992:dubois::upper,1998:moral:delsagrado:aggregation,2006:troffaes:conjunc:rule}.
We could also
study if the results can be extended to infinite possibility spaces;
although clearly the game-theoretic interpretation
may prove problematic in this respect.
Finally,
many other imprecise probability models,
such as belief functions, probability boxes, and so on,
might benefit from similar studies.

\section*{Acknowledgements}
The research in this paper has been supported by project
MTM2010-17844 and by project Labex MS2T (Reference ANR-11-IDEX-0004-02).

\bibliographystyle{plainurl}
\bibliography{all}

\begin{thebibliography}{10}

\bibitem{2009:alvarez}
Diego~A. Alvarez.
\newblock A {M}onte {C}arlo-based method for the estimation of lower and upper
  probabilities of events using infinite random sets of indexable type.
\newblock {\em Fuzzy Sets and Systems}, 160(3):384--401, 2009.
\newblock \href {http://dx.doi.org/10.1016/j.fss.2008.08.006}
  {\path{doi:10.1016/j.fss.2008.08.006}}.

\bibitem{2007:baudrit}
C{\'e}dric Baudrit, Dominique Guyonnet, and Didier Dubois.
\newblock Joint propagation of variability and imprecision in assessing the
  risk of groundwater contamination.
\newblock {\em Journal of contaminant hydrology}, 93(1):72--84, 2007.
\newblock \href {http://dx.doi.org/10.1016/j.jconhyd.2007.01.015}
  {\path{doi:10.1016/j.jconhyd.2007.01.015}}.

\bibitem{2001:bemporad::convexity}
Alberto Bemporad, Komei Fukuda, and Fabio~D. Torrisi.
\newblock Convexity recognition of the union of polyhedra.
\newblock {\em Computational Geometry}, 18(3):141--154, 2001.
\newblock \href {http://dx.doi.org/10.1016/S0925-7721(01)00004-9}
  {\path{doi:10.1016/S0925-7721(01)00004-9}}.

\bibitem{2010:benavoli:aggregation}
Alessio Benavoli and Alessandro Antonucci.
\newblock An aggregation framework based on coherent lower previsions:
  Application to zadeh's paradox and sensor networks.
\newblock {\em International Journal of Approximate Reasoning},
  51(9):1014--1028, 2010.
\newblock \href {http://dx.doi.org/10.1016/j.ijar.2010.08.010}
  {\path{doi:10.1016/j.ijar.2010.08.010}}.

\bibitem{1994:chateauneuf::combination}
Alain Chateauneuf.
\newblock Combination of compatible belief functions and relation of
  specificity.
\newblock In {\em Advances in the {D}empster-{S}hafer theory of evidence},
  pages 97--114. Wiley, 1994.

\bibitem{1994:coolen}
Frank Coolen.
\newblock {\em Statistical modelling of experts opinions using imprecise
  probabilities}.
\newblock PhD thesis, Technical University of Eindhoven, 1994.

\bibitem{1987:decampos}
Luis~M. de~Campos.
\newblock {\em Caracterizaci{\'o}n y estudio de medidas e integrales difusas a
  partir de probabilidades}.
\newblock PhD thesis, University of Granada, 1987.

\bibitem{1997:decooman:possibility1}
Gert de~Cooman.
\newblock Possibility theory {I}: the measure- and integral-theoretic
  groundwork.
\newblock {\em International Journal of General Systems}, 25:291--323, 1997.
\newblock \href {http://dx.doi.org/10.1080/03081079708945160}
  {\path{doi:10.1080/03081079708945160}}.

\bibitem{2005:decooman::behavioural}
Gert De~Cooman.
\newblock A behavioural model for vague probability assessments.
\newblock {\em Fuzzy sets and systems}, 154(3):305--358, 2005.
\newblock \href {http://dx.doi.org/10.1016/j.fss.2005.01.005}
  {\path{doi:10.1016/j.fss.2005.01.005}}.

\bibitem{1999:decooman:aeyels::sup:pres:upp:prob}
Gert {de Cooman} and Dirk Aeyels.
\newblock Supremum preserving upper probabilities.
\newblock {\em Information Sciences}, 118:173--212, 1999.
\newblock \href {http://dx.doi.org/10.1016/S0020-0255(99)00007-9}
  {\path{doi:10.1016/S0020-0255(99)00007-9}}.

\bibitem{2004:decooman:troffaes::products:aggregation}
Gert {d}e Cooman and Matthias C.~M. Troffaes.
\newblock Coherent lower previsions in systems modelling: products and
  aggregation rules.
\newblock {\em Reliability Engineering and System Safety}, 85(1--3):113--134,
  2004.
\newblock \href {http://dx.doi.org/10.1016/j.ress.2004.03.007}
  {\path{doi:10.1016/j.ress.2004.03.007}}.

\bibitem{2011:destercke}
S\'{e}bastien Destercke and Didier Dubois.
\newblock Idempotent conjunctive combination of belief functions: extending the
  idempotent rule of possibility theory.
\newblock {\em Information Sciences}, 181:3925--3945, 2011.
\newblock \href {http://dx.doi.org/10.1016/j.ins.2011.05.007}
  {\path{doi:10.1016/j.ins.2011.05.007}}.

\bibitem{2009:destercke}
S{\'e}bastien Destercke, Didier Dubois, and Eric Chojnacki.
\newblock A consonant approximation of the product of independent consonant
  random sets.
\newblock {\em International Journal of Uncertainty, Fuzziness and
  Knowledge-Based Systems}, 17(06):773--792, 2009.
\newblock \href {http://dx.doi.org/10.1142/S0218488509006261}
  {\path{doi:10.1142/S0218488509006261}}.

\bibitem{1986:dubois}
Didier Dubois and Henri Prade.
\newblock A set theoretic view of belief functions.
\newblock {\em International Journal of General Systems}, 12:193--226, 1986.
\newblock \href {http://dx.doi.org/10.1007/978-3-540-44792-4_14}
  {\path{doi:10.1007/978-3-540-44792-4_14}}.

\bibitem{1988:dubois:possibility}
Didier Dubois and Henri Prade.
\newblock {\em Possibility Theory -- An Approach to Computerized Processing of
  Uncertainty}.
\newblock Plenum Press, New York, 1988.

\bibitem{1988:dubois::comb}
Didier Dubois and Henri Prade.
\newblock Representation and combination of uncertainty with belief functions
  and possibility measures.
\newblock {\em Computational Intelligence}, 4:244--264, 1988.
\newblock \href {http://dx.doi.org/10.1111/j.1467-8640.1988.tb00279.x}
  {\path{doi:10.1111/j.1467-8640.1988.tb00279.x}}.

\bibitem{1992:dubois::upper}
Didier Dubois and Henri Prade.
\newblock When upper probabilities are possibility measures.
\newblock {\em Fuzzy Sets and Systems}, 49(1):65--74, 1992.
\newblock \href {http://dx.doi.org/10.1016/0165-0114(92)90110-P}
  {\path{doi:10.1016/0165-0114(92)90110-P}}.

\bibitem{1999:dubois}
Didier Dubois, Henri Prade, and Ron Yager.
\newblock Merging fuzzy information.
\newblock In H.~Prade J.~C.~Bezdek, D.~Didier, editor, {\em Fuzzy sets in
  approximative reasoning and information systems}, Handbook of Fuzzy Sets,
  chapter~6, pages 335--401. Kluwer, Dordrecht, 1999.

\bibitem{1982:giles:semantics}
Robin Giles.
\newblock Semantics for fuzzy reasoning.
\newblock {\em International Journal of Man-Machine Studies}, 17(4):401--415,
  1982.
\newblock \href {http://dx.doi.org/10.1016/S0020-7373(82)80041-2}
  {\path{doi:10.1016/S0020-7373(82)80041-2}}.

\bibitem{2008:hunter}
Anthony Hunter and Weiru Liu.
\newblock A context-dependent algorithm for merging uncertain information in
  possibility theory.
\newblock {\em IEEE Transactions on Systems Man and Cybernetics: part A},
  38(6):1385--1397, 2008.
\newblock \href {http://dx.doi.org/10.1109/TSMCA.2008.2003457}
  {\path{doi:10.1109/TSMCA.2008.2003457}}.

\bibitem{1950:kolmogorov}
Andrei~N. Kolmogorov.
\newblock {\em Foundations of the Theory of Probability}.
\newblock Chelsea Publishing Company, New York, 1950.

\bibitem{2006:kroupa}
Tom{\'a}{\v{s}} Kroupa.
\newblock How many extreme points does the set of probabilities dominated by a
  possibility measure have.
\newblock In {\em Proceedings of 7th Workshop on Uncertainty Processing WUPES},
  volume~6, pages 89--95, 2006.

\bibitem{1957:luce:raiffa}
R.~Duncan Luce and Howard Raiffa.
\newblock {\em Games and Decisions: Introduction and Critical Survey}.
\newblock Dover Publications, 1957.

\bibitem{2008:miranda::survey:lowprevs}
Enrique Miranda.
\newblock A survey of the theory of coherent lower previsions.
\newblock {\em International Journal of Approximate Reasoning}, 48(2):628--658,
  2008.
\newblock \href {http://dx.doi.org/10.1016/j.ijar.2007.12.001}
  {\path{doi:10.1016/j.ijar.2007.12.001}}.

\bibitem{2003:miranda::extreme}
Enrique Miranda, In{\'e}s Couso, and Pedro Gil.
\newblock Extreme points of credal sets generated by 2-alternating capacities.
\newblock {\em International Journal of Approximate Reasoning}, 33(1):95--115,
  2003.
\newblock \href {http://dx.doi.org/10.1016/S0888-613X(02)00149-4}
  {\path{doi:10.1016/S0888-613X(02)00149-4}}.

\bibitem{2003:miranda::epistemic}
Enrique Miranda and Gert De~Cooman.
\newblock Epistemic independence in numerical possibility theory.
\newblock {\em International journal of approximate reasoning}, 32(1):23--42,
  2003.
\newblock \href {http://dx.doi.org/10.1016/S0888-613X(02)00087-7}
  {\path{doi:10.1016/S0888-613X(02)00087-7}}.

\bibitem{1998:moral:delsagrado:aggregation}
S.~Moral and J.~{del Sagrado}.
\newblock Aggregation of imprecise probabilities.
\newblock In B.~Bouchon-Meunier, editor, {\em Aggregation and Fusion of
  Imperfect Information}, pages 162--188. Physica-Verlag, New York, 1998.

\bibitem{2010:palacios}
Ana~M. Palacios, Luciano S{\'a}nchez, and In{\'e}s Couso.
\newblock Diagnosis of dyslexia with low quality data with genetic fuzzy
  systems.
\newblock {\em International Journal of Approximate Reasoning},
  51(8):993--1009, 2010.
\newblock \href {http://dx.doi.org/10.1016/j.ijar.2010.07.008}
  {\path{doi:10.1016/j.ijar.2010.07.008}}.

\bibitem{1995:sandri}
Sandra~A. Sandri, Didier Dubois, and Henk~W. Kalfsbeek.
\newblock Elicitation, assessment, and pooling of expert judgments using
  possibility theory.
\newblock {\em IEEE Transactions on Fuzzy Systems}, 3(3):313--335, 1995.
\newblock \href {http://dx.doi.org/10.1109/91.413236}
  {\path{doi:10.1109/91.413236}}.

\bibitem{1961:smith}
Cedric A.~B. Smith.
\newblock Consistency in statistical inference and decision.
\newblock {\em Journal of the Royal Statistical Society}, B(23):1--37, 1961.
\newblock URL: \url{http://www.jstor.org/stable/2983842}.

\bibitem{2006:troffaes:conjunc:rule}
Matthias C.~M. Troffaes.
\newblock Generalising the conjunction rule for aggregating conflicting expert
  opinions.
\newblock {\em International Journal of Intelligent Systems}, 21(3):361--380,
  Mar 2006.
\newblock \href {http://dx.doi.org/10.1002/int.20140}
  {\path{doi:10.1002/int.20140}}.

\bibitem{2014:troffaes:decooman::lower:previsions}
Matthias C.~M. Troffaes and Gert de~Cooman.
\newblock {\em Lower Previsions}.
\newblock Wiley Series in Probability and Statistics. Wiley, 2014.

\bibitem{2013:troffaes:infsci:pbox:possib}
Matthias C.~M. Troffaes, Enrique Miranda, and S{\'e}bastien Destercke.
\newblock On the connection between probability boxes and possibility measures.
\newblock {\em Information Sciences}, 224:88--108, March 2013.
\newblock \href {http://arxiv.org/abs/1103.5594} {\path{arXiv:1103.5594}},
  \href {http://dx.doi.org/10.1016/j.ins.2012.09.033}
  {\path{doi:10.1016/j.ins.2012.09.033}}.

\bibitem{1981:walley:lowuppprobs}
Peter Walley.
\newblock Coherent lower (and upper) probabilities.
\newblock Technical report, University of Warwick, Coventry, 1981.
\newblock Statistics Research Report 22.

\bibitem{1982:walley:aggregation}
Peter Walley.
\newblock The elicitation and aggregation of beliefs.
\newblock Technical report, University of Warwick, Coventry, 1982.
\newblock Statistics Research Report 23.

\bibitem{1991:walley}
Peter Walley.
\newblock {\em Statistical Reasoning with Imprecise Probabilities}.
\newblock Chapman and Hall, London, 1991.

\bibitem{1996:walley::uncertaintyinexpertsystems}
Peter Walley.
\newblock Measures of uncertainty in expert systems.
\newblock {\em Artificial Intelligence}, 83:1--58, 1996.
\newblock \href {http://dx.doi.org/10.1016/0004-3702(95)00009-7}
  {\path{doi:10.1016/0004-3702(95)00009-7}}.

\bibitem{1983:yager}
Ron Yager.
\newblock Entropy and specificity in a mathematical theory of evidence.
\newblock {\em International Journal of General Systems}, 9:249--260, 1983.
\newblock \href {http://dx.doi.org/10.1080/03081078308960825}
  {\path{doi:10.1080/03081078308960825}}.

\bibitem{1978:zadeh::possibility}
Lofti~A. Zadeh.
\newblock Fuzzy sets as a basis for a theory of possibility.
\newblock {\em Fuzzy Sets and Systems}, 1:3--28, 1978.
\newblock \href {http://dx.doi.org/10.1016/0165-0114(78)90029-5}
  {\path{doi:10.1016/0165-0114(78)90029-5}}.

\bibitem{2013:zaffalon::probtime}
Marco Zaffalon and Enrique Miranda.
\newblock Probability and time.
\newblock {\em Artificial Intelligence}, 198:1--51, May 2013.
\newblock \href {http://dx.doi.org/10.1016/j.artint.2013.02.005}
  {\path{doi:10.1016/j.artint.2013.02.005}}.

\end{thebibliography}

\end{document}